\documentclass[11pt,tbtags]{amsart}
\usepackage{amssymb}
\usepackage{amsmath}
\usepackage{amsthm}
\usepackage[swedish, english]{babel}
\usepackage[latin1]{inputenc}
\usepackage{psfrag}
\usepackage{graphicx}
\usepackage[square, comma, numbers, sort&compress]{natbib}
\usepackage{subfigure}
\usepackage{nicefrac}
\usepackage{color}
\usepackage[left=3.1cm,right=3.1cm,top=3cm,bottom=3cm]{geometry}

\theoremstyle{plain}
\newcommand{\E}{\mathbb E}
\newcommand{\R}{\mathbb R}
\newcommand{\C}{\mathcal C}
\newcommand{\D}{\mathcal D}

\def\P{{\mathbb P}}

\newenvironment{remark}[1][Remark]{\begin{trivlist}
\item[\hskip \labelsep {\bf Remark}]}{\end{trivlist}}
\newtheorem{theorem}{Theorem}[section]
\newtheorem{lemma}[theorem]{Lemma}
\newtheorem{corollary}[theorem]{Corollary}
\newtheorem{proposition}[theorem]{Proposition}

\theoremstyle{definition}

\title[Optimally Stopping a BB with an Unknown Pinning Time]{Optimally Stopping a Brownian Bridge with an \\Unknown Pinning Time: A Bayesian Approach}
\author[Kristoffer Glover]{Kristoffer Glover}
\address{University of Technology Sydney, P.O. Box 123, Broadway, NSW 2007, Australia.}
\date{\today}

\begin{document}

\begin{abstract}
We consider the problem of optimally stopping a Brownian bridge with an unknown pinning time so as to maximise the value of the process upon stopping.
Adopting a Bayesian approach, we assume the stopper has a general continuous prior and is allowed to update their belief about the value of the pinning time through sequential observations of the process.
Uncertainty in the pinning time influences both the conditional dynamics of the process and the expected (random) horizon of the optimal stopping problem.
We analyse certain gamma and beta distributed priors in detail.
Remarkably, the optimal stopping problem in the gamma case becomes \emph{time homogeneous} and is completely solvable in closed form.
Moreover, in the beta case we find that the optimal stopping boundary takes on a \emph{square-root} form, similar to the classical solution with a known pinning time.
\end{abstract}

\maketitle

\section{Introduction}
The problem of stopping a Brownian bridge so as to maximise the expected value upon stopping has a long and rich history in the field of optimal stopping.
As a continuous-time analogue of \cite{CR}, it was first considered in \cite{D} and was solved explicitly in the seminal work of Larry Shepp \cite{S}. 
It can also be seen as a special case of the problem studied in \cite{F}.
The problem is quite remarkable in that it is one of the very few finite-horizon optimal stopping problems to yield an explicit solution.
The same problem (amongst others) was also solved in \cite{EW} using an alternative (Markovian) approach and other notable work on this problem includes \cite{ES}, who provided yet another alternative proof of the result of \cite{S}.
Further, \cite{BCSY} extended the work of \cite{EW} to consider the problem of maximising the expected spread between the value of the process between two stopping times.
Once more the authors were able to obtain explicit solutions to this more complicated problem.

More recent work has introduced uncertainty in the pinning level of the Brownian bridge.
For example, \cite{EV2} considered the problem of stopping a Brownian bridge with an unknown pinning point so as to maximise the expected value upon stopping.
The authors allowed for a general prior distribution of the unknown pinning point and revealed a rich structure of the optimal stopping region even in the simple case of a two-point distribution.
Similar optimal trading problems were also considered in \cite{CJK} and \cite{LLL}, who used an exponential randomized Brownian bridge to model asset price dynamics under a trader's subjective market view.
Once more a rich solution structure was found by these authors, with disconnected continuation/exercise regions.

In contrast to uncertainty in the pinning level, the present paper considers uncertainty in the \emph{time} at which the Brownian bridge pins.
Such uncertainty for a Brownian bridge has recently been introduced by Bedini, Buckdahn, and Engelbert in \cite{BBE}, who considered a model of a Brownian bridge on a random time interval to model the flow of information about a company's default.
These authors outlined the basic properties of such processes and presented many useful results.
However, no optimal stopping problems were considered in \cite{BBE}.
To the best of our knowledge, this is the first paper to consider the problem of stopping a Brownian bridge with an uncertain pinning time.

When optimally stopping such a Brownian bridge, the random pinning time enters into the problem in two important, but distinct, ways.
Firstly, the drift of the process depends of the value of the unknown pinning time, and so observations of the sample path can be used to filter information about its true value.
Hence, when adopting a Bayesian approach, this updating introduces a more complicated dynamic into the optimal stopping problem at hand.
Secondly, the unknown pinning time effectively introduces a random time horizon into the optimal stopping problem, since stopping should ideally happen at or before the random pinning time.
Moreover, the random horizon introduced in this way may be considered non-standard in the sense that information about the horizon's true value is revealed through sequential observation of the underlying dynamics.
Other optimal stopping problems in the literature with a random horizon typically make the assumption that the horizon is drawn from a random variable (independent to the underlying Brownian motion) but that the horizon does not enter into the dynamics of $X$ in any way.
In such cases, no inference about the realization of the random horizon can be made before it happens.
Nor is the random horizon in our setting simply the first hitting time of some known boundary, another common mechanism used in the literature for randomly terminating a stopping problem.
In fact, we demonstrate that the random horizon in the present setting can be represented by elastic killing of the process at zero with a time-dependent killing rate (dependent on the assumed prior distribution of the optimal stopper).

In this paper we consider the optimal stopping problem for an agent with a general continuous prior distribution of the unknown pinning time.
Similar to previous studies, the structural properties of the optimal stopping region are dependent on the chosen prior.
Certain gamma and beta distributed priors are examined in detail and various properties of the solution in these cases are presented.
Specifically, we document the remarkable property that for a particular gamma distribution the stopper's conditional estimate of the pinning time is such that the conditional dynamics of the underlying Brownian bridge becomes time homogeneous.
Moreover, the optimal stopping problem for such a prior is shown to be completely solvable in closed-form.
We also show that for a certain class of beta distributions the optimal stopping boundary is of a \emph{square-root} form, similar to the classical solution when the pinning time is known with certainty.

It is well known that the Brownian bridge appears as the large population limit of the cumulative sum process when sampling without replacement from a finite population (see \cite{R}).
As such, and as was noted in \cite{S} and \cite{EW} amongst others, the problem of maximising a stopped Brownian bridge can be thought of as a continuous analog of the following urn problem.
$2n$ balls are drawn randomly and sequentially from an urn containing an equal number ($n$) of red and black balls.
If every red ball wins you a dollar and every black ball loses you a dollar, and you could stop the game at any time, what would your optimal strategy be so as to maximise your expected profit?
The classical problem considered in \cite{S} corresponds to the situation in which the number $n$ is known.
The problem in the current article, however, corresponds to a situation where the stopper does not know $n$ with certainty, but has a prior belief about its value and updates this belief as balls are subsequently drawn.
Intuitively, if more balls of a given colour have already been drawn, then drawing an additional ball of the same colour would suggest that the total number of balls in the urn is larger than previously believed.
It is important therefore that such learning be incorporated into the optimal stopping strategy, complicating the analysis somewhat.

Brownian bridges also play a key role in many areas of statistics and probability theory and have found use in many applications across numerous fields.
In addition to its appearance in the large population limit of the cumulative sum process mentioned above, it also appears in the Kolmogorov-Smirnov test for the equality of two distributions.
In finance, they have arisen in the modelling of the so-called stock pinning effect (see \cite{AL,JIS}), the modelling of arbitrage dynamics (see \cite{BS1,LL}), and as the equilibrium price dynamics in a classical model of insider trading (see \cite{Ba,K}).

Outside of the Brownian bridge, there has also been numerous examples in the extant literature considering optimal stopping problems with incomplete information about the underlying stochastic process.
For example, optimal liquidation problems with an unknown drift were studied in \cite{EL} and \cite{EV}, and with an unknown jump intensity in \cite{Lu}.
In the context of American-style option valuation, the effect of incomplete information on optimal exercise was also considered in \cite{DMV}, \cite{EV3} and \cite{Ga}.
All of the above examples consider incomplete information about the parameters of a time-homogenous process.
Here, however, the underlying dynamics are time inhomogeneous.

The remainder of this article is structured as follows.
In Section \ref{sec:genprior} we formulate the optimal stopping problem under the assumption of a general prior distribution for the pinning time and investigate various structural properties of the solution.
The cases of a gamma and beta distributed prior are studied in further detail in Sections \ref{sec:gamma} and \ref{sec:beta}, respectively.

\section{Problem formulation and filtering assumptions}\label{sec:genprior}

1. Let $X=(X_t)_{t\geq 0}$ be a Brownian bridge that pins to zero (without loss of generality) at some strictly positive time $\theta$.
The Brownian bridge $X$ is therefore known to solve the following stochastic differential equation
\begin{equation}\label{eqn:process}
  dX_t =-\frac{X_t}{\theta-t}dt+dB_t, \quad X_0=\kappa\in\R
\end{equation}
for $t\in[0,\theta)$ and where $B=(B_t)_{t\geq 0}$ is a standard Brownian motion defined on a probability space $(\Omega,\mathcal{F},\P)$.
We assume that $\theta$ is unknown to the optimal stopper but that they are able to glean information about its true value through continuous observation of the process $X$. 
As such, we adopt a Bayesian approach in which the stopper has some prior belief about the pinning time, denoted $\mu$, and updates this belief (via Bayes) over time.
We further assume that $\theta$ is independent of $B$ under $\P$ and that $X_t=0$ for all $t\geq\theta$.
Such random horizon Brownian bridges have recently been studied in detail by \cite{BBE}, where it was shown (in Corollary 6.1) that the completed natural filtration generated by $X$ satisfied the usual conditions.
It was also shown (in Proposition 3.1) that $\theta$ is a stopping time with respect to this filtration.
Hence we let the filtration used by the stopper be the filtration generated by $X$, denoted $\mathcal{F}^X$.
To avoid certain technical issues, in this article we will assume that the prior distribution $\mu$ has a finite first moment and admits a continuous density with $\mathrm{supp}\,\mu\subseteq[0,T]$.
Note that $T$ can be either finite or infinite and indeed we will consider specific examples of both cases in the following sections.

%

\vspace{4pt}

2. The problem under investigation is to find the optimal stopping strategy that maximises the expected value of $X$ upon stopping, i.e.
\begin{equation}\label{eqn:prob}
  V=\sup_{\tau\geq 0}\E\left[X_\tau\right],
\end{equation}
where the supremum is taken over all $\mathcal{F}^X$-stopping times.
In fact, since $X_t=0$ for all $t\geq\theta$, any stoping time $\tau$ can be replaced with $\tau\wedge\theta$ without loss of generality (since the expected payoff would be the same otherwise).

\vspace{4pt}

3. Next, recall that under the Bayesian approach the stopper will update their belief about the pinning time given continuous observation of the process $X$.
The details of this updating has recently been provided in \cite{BBE} which motivates the following result.

\begin{proposition}\label{prop:conditioning}
Let $q:\mathbb{R}_+\mapsto\mathbb{R}$ satisfy $\int_{0}^{\infty}|q(r)|\mu(dr)<\infty$. Then
\begin{equation}\label{eqn:project}
  \mathbb{E}\left[q(\theta)\,|\,\mathcal{F}_t^X\right]=q(\theta)\mathbf{1}_{\{t\geq\theta\}}+\int_t^{\infty}q(r)\mu_{t,X_t}(dr)\mathbf{1}_{\{t<\theta\}}
\end{equation}
for any $t>0$ and where
\begin{equation}
  \mu_{t,x}(dr):=\frac{\sqrt{\frac{r}{r-t}}e^{-\frac{r\left(x-\kappa+t\kappa/r\right)^2}{2t(r-t)}}\mu(dr)}{\int_t^{\infty}\sqrt{\frac{u}{u-t}}e^{-\frac{u\left(x-\kappa+t\kappa/u\right)^2}{2t(u-t)}}\mu(du)}.
\end{equation}
\end{proposition}

\begin{proof}
The proof in the case of $\kappa=0$ is given in the proof of Corollary 4.1 in \cite{BBE}.
To extend these arguments to a nonzero starting value we exploit the results of conditioning Brownian motion at time $t$ on the knowledge of its value at both an earlier and later time (cf. \cite[pp.~116-117]{E}) given by
  \[(X_t\,|\,X_s=\kappa,X_r=y)\sim\mathcal{N}\left(\frac{r-t}{r-s}\kappa+\frac{t-s}{r-s}y,\frac{(r-t)(t-s)}{r-s}\right)\]
where $s<t<r$. Therefore, setting $s=y=0$ we arrive at the desired density, which can be used in place of the $\kappa=0$ case in the proof of Corollary 4.1 in \cite{BBE}.
\end{proof}

With this result in place we can obtain the dynamics of $X$ adapted to $\mathcal{F}^X$ as follows.

\begin{proposition}\label{prop:dynamics}
For $t<\theta$, the dynamics of $X$ can be written as
\begin{equation}\label{eqn:X}
  dX_t = -X_t f(t,X_t)dt+d\bar{B}_t
\end{equation}
where $\bar{B}=(\bar{B}_t)_{t\geq 0}$ is a $\mathcal{F}^X$-Brownian motion and
  \[f(t,X_t):=\mathbb{E}\Big[\frac{1}{\theta-t}\,|\,\mathcal{F}_t^X,t<\theta\Big]\]
which can be expressed as
\begin{equation}\label{eqn:f}
  f(t,x)=\frac{\int_t^{\infty}\frac{1}{r-t}\sqrt{\frac{r}{r-t}}e^{-\frac{r\left(x-\kappa+t\kappa/r\right)^2}{2t(r-t)}}\mu(dr)}{\int_t^{\infty}\sqrt{\frac{r}{r-t}}e^{-\frac{r\left(x-\kappa+t\kappa/r\right)^2}{2t(r-t)}}\mu(dr)}=\int_{t}^{\infty}\frac{1}{r-t}\,\mu_{t,x}(dr).
\end{equation}
For $t\geq\theta$, we have $X_t=0$.
\end{proposition}

\begin{proof}
Given the dynamics in \eqref{eqn:process} we define $Z_t:=\mathbb{E}\bigl[\frac{1}{\theta-t}\,|\,\mathcal{F}_t^X\bigr]$ and the process
  \[\bar{B}_t:=\int_0^t\left(Z_s-\frac{1}{\theta-s}\right)X_s \,ds+B_t,\]
for $t\in[0,\theta)$.
To show that $\bar{B}$ is also an $\mathcal{F}^X$-Brownian motion we have the following arguments
\begin{align}
\bar{B}_{t+s} &= \bar{B}_t+\int_{t}^{t+s}\left(Z_u-\frac{1}{\theta-u}\right)X_u \,du+\int_{t}^{t+s}dB_u \nonumber \\
\Rightarrow \E\left[\bar{B}_{t+s}\,|\mathcal{F}_{t}^{X}\right] &= \bar{B}_t+\E\left[\int_t^{t+s}\left(Z_u-\frac{1}{\theta-u}\right)X_u \, du\,|\,\mathcal{F}_{t}^{X}\right]+\E\left[B_{t+s}-B_t\,|\,\mathcal{F}_t^X\right]\nonumber \\
&= \bar{B}_t+\int_t^{t+s}\E\left[\E\left[\left(Z_u-\frac{1}{\theta-u}\right)X_u\,|\,\mathcal{F}_{u}^{X}\right]\,|\,\mathcal{F}_{t}^{X}\right]du \nonumber \\
&= \bar{B}_t+\int_t^{t+s}E\left[Z_u X_u-X_u\E\left[\frac{1}{\theta-u}\,|\,\mathcal{F}_{u}^{X}\right]\,|\,\mathcal{F}_{t}^{X}\right]du \nonumber \\
&= \bar{B}_t, \nonumber
\end{align}
and clearly $\langle\bar{B}\rangle_t=t$.
Hence we have
  \[dX_t = -X_t Z_tdt+d\bar{B}_t=-X_t f(t,X_t)dt+d\bar{B}_t\]
for $t<\theta$. 
Next, to determine the required expression for $f(t,x)$ we would like to use the results of Proposition \ref{prop:conditioning} after setting $q(r)=1/(r-t)$.
Unfortunately, for this particular choice of $q$ the required integrability condition $\int_{0}^{\infty}|q(r)|\mu(dr)<\infty$ is not satisfied in general, nor for our specific choices of $\mu$ made below.
However, it can be shown that \eqref{eqn:project} is still valid for $q(r)=1/(r-t)$ by applying Proposition \ref{prop:conditioning} to a truncated version of this function and then passing to the limit.
Specifically, let $q_\epsilon(r)=\frac{1}{r-t}\mathbf{1}_{\{|r-t|>\epsilon\}}$ for some $\epsilon>0$ and assume that $\mu$ is such that $\int_{0}^{\infty}|q_{\epsilon}(r)|\mu(dr)<\infty$.
We can apply Proposition \ref{prop:conditioning} to $q_{\epsilon}$ and take the limit as $\epsilon\downarrow 0$ to yield
\begin{align}
  \mathbb{E}\Big[\frac{1}{\theta-t}\,|\,\mathcal{F}_t^X\Big]=\lim_{\epsilon\downarrow 0}\mathbb{E}\left[q_{\epsilon}(\theta)\,|\,\mathcal{F}_t^X\right] &= \lim_{\epsilon\downarrow 0}q_{\epsilon}(\theta)\mathbf{1}_{\{t\geq\theta\}}+\int_t^{\infty}\lim_{\epsilon\downarrow 0}q_{\epsilon}(r)\mu_{t,X_t}(dr)\mathbf{1}_{\{t<\theta\}} \nonumber \\
  &= \frac{1}{\theta-t}\mathbf{1}_{\{t\geq\theta\}}+\int_t^{\infty}\frac{1}{r-t}\mu_{t,X_t}(dr)\mathbf{1}_{\{t<\theta\}} \nonumber
\end{align}
obtaining the desired expression for $f(t,x)$ and completing the proof.
\end{proof}

4. It is clear from its definition that $f\geq 0$ (since $\theta\geq t$) and it can also be shown to have the following property.

\begin{proposition}\label{prop:fdecrease}
Given $f$ as defined in \eqref{eqn:f} we have that $x\mapsto f(t,x)$ is increasing for $x<0$ and decreasing for $x>0$ for any $t>0$.
\end{proposition}

\begin{proof}
Straightforward differentiation of \eqref{eqn:f} yields
\begin{align}
  \frac{\partial f(t,x)}{\partial x} &= -\frac{x}{t}\left(\int_t^\infty\frac{r}{(r-t)^2}\mu_{t,x}(dr)-\int_t^\infty\frac{1}{r-t}\mu_{t,x}(dr)\int_t^\infty\frac{r}{r-t}\mu_{t,x}(dr)\right) \nonumber \\
  &= -\frac{x}{t}\left(\E\Bigl[\frac{\theta}{(\theta-t)^2}\,|\,\mathcal{F}_t^X\Bigr]-\E\Bigl[\frac{1}{\theta-t}\,|\,\mathcal{F}_t^X\Bigr]\E\Bigl[\frac{\theta}{\theta-t}\,|\,\mathcal{F}_t^X\Bigr]\right). \nonumber \\
  &= -\frac{x}{t^2}\textrm{Var}\Bigl(\frac{1}{\theta-t}\,|\,\mathcal{F}_t^X\Bigr) \nonumber
\end{align}
where the last equality is a consequence of noting that $\tfrac{\theta}{\theta-t}=\tfrac{1}{t}+\tfrac{1}{t(\theta-t)}$. We thus observe that $\textrm{Sign}\left(f_x(t,x)\right)=\textrm{Sign}(-x)$, completing the proof.
\end{proof}

\begin{remark}
Proposition \ref{prop:fdecrease} demonstrates the intuitive result that a movement of $X$ away from zero gives information that the process is more likely to pin at a \emph{later} time, i.e.~that $\theta$ is larger and hence $1/(\theta-t)$ is smaller.
In other words, learning about the unknown pinning time produces a decreased pinning force as $X$ moves away from zero.
However, the $-X_t$ term in the drift of \eqref{eqn:X} will result in an increased pinning force as $X$ deviates from zero.
The overall affect of these two competing contributions to the drift is not clear in general and indeed different priors can result in different behaviour of the function $-xf(t,x)$ and hence the optimal stopping strategy in \eqref{eqn:prob}.
\end{remark}

\vspace{4pt}

We also observe the following general properties of the function $f$, which will be seen in the specific examples considered later.
Firstly, Proposition \ref{prop:fdecrease} implies that the drift of the SDE in \eqref{eqn:X} satisfies the (sublinear) growth condition $\|xf(t,x)\|\leq k(1+\|x\|)$ for some positive constant $k$.
Therefore it is known that \eqref{eqn:X} admits a weak solution (cf. Proposition 3.6 in \cite[][p.~303]{KS}).
Secondly, since $(r-t)^{-3/2}$ is not integrable at $r=t$ but $(r-t)^{-1/2}$ is, it can be seen that $f(t,0)=\infty$ at time points for which $\mu$ has a strictly positive density.
Otherwise, $f(t,0)<\infty$.
Thirdly, in the cases where $f(t,0)=\infty$, the drift function $-xf(t,x)$ will be seen to have a non-zero (but finite) limit as $|x|\to 0$.
This limit appears difficult to analyse in general from \eqref{eqn:f} but can be seen clearly in the examples considered in Sections \ref{sec:gamma} and \ref{sec:beta}.
A consequence of this non-zero limit is that a discontinuity appears in $-xf(t,x)$ at $x=0$ (since $x\mapsto f(t,x)$ is even, hence $x\mapsto -xf(t,x)$ odd).

\vspace{4pt}

5. Next, while \eqref{eqn:X} describes the conditional dynamics of the process up to the random pinning time $\theta$, it does not tell us when this time will actually occur.
This additional randomness must be incorporated into the expectation taken in \eqref{eqn:prob} and to do so we assume from now on (without loss of generality) that the process $X$ is killed at the random time $\theta$ (and sent to some cemetery state if desired).
Since $X_{\theta}=0$, the process must be at zero when killed, however the process can visit zero many times before being killed there.
Hence the killing at zero is \emph{elastic}, in the sense that killing happens only at some `rate'.
When $\mu$ admits a continuous density, this rate is given by the following result.

\begin{proposition}\label{lem:killing}
If the distribution of $\theta$ admits a continuous density function with respect to Lebesgue, denoted by $\mu(\cdot)$ and with $\mathrm{supp}\,\mu\subseteq[0,T]$, then the \emph{infinitesimal killing rate} for the process $X$ is given by
\begin{equation}\label{eqn:kill}
  c(t,x)=q(t)\delta_0(x) \,\,\textrm{ where }\,\, q(t):=\frac{\mu(t)}{\frac{1}{\sqrt{2\pi t}}\int_{t}^{T}\sqrt{\frac{r}{r-t}}\,\mu(r)dr}
\end{equation}
and $\delta_0(x)$ denotes the Dirac delta function of $x$ at zero.
\end{proposition}

\begin{proof}
Identifying $\theta$ as the (random) lifetime of the process $X$ it is well known (see, for example, \cite[][p.~130]{La}) that the infinitesimal killing rate is given by
\begin{equation}\label{eqn:killdef}
  c(t,x)=\lim_{\epsilon\downarrow 0}\left\{\frac{1}{\epsilon}\,\P\left(t<\theta\leq t+\epsilon\,|\,\mathcal{F}_t^X\right)\right\}.
\end{equation}
To obtain the probability required above we note that, under the assumption that $\theta$ has a continuous density with respect to Lebesgue, Theorem 3.2 in \cite{BBE2} states that the compensator of the indicator process $\mathbf{1}_{\{\theta\leq t\}}$ admits the representation
\begin{equation}\label{eqn:compK}
  K_t=\int_0^{t\wedge\theta}q(s)d\ell_s^0(X),
\end{equation}
where $q$ is as defined in \eqref{eqn:kill} and $\ell_s^0(X)$ denotes the local time at zero of the process $X$ up to time $s$.
It thus follows that
\begin{align}
  \P(t<\theta\leq t+\epsilon\,|\,\mathcal{F}_t^X) &= \E[\mathbf{1}_{\{\theta\leq t+\epsilon\}}-\mathbf{1}_{\{\theta\leq t\}}\,|\,\mathcal{F}_t^X] \nonumber \\
  &= \E[K_{t+\epsilon}-K_t\,|\,\mathcal{F}_t^X] \nonumber \\
  &= \E\int_{t\wedge\theta}^{(t+\epsilon)\wedge\theta}q(s)d\ell_s^0(X) \nonumber \\
  &= \E\int_{0}^{\epsilon\wedge(\theta-t)}q(t+u)d\ell_u^0(X). \label{eqn:probpin}
\end{align}
Using the identity \eqref{eqn:probpin} in \eqref{eqn:killdef} therefore yields
\begin{align}
  c(t,x) &= \lim_{\epsilon\downarrow 0}\left\{\frac{1}{\epsilon}\,\E\int_{0}^{\epsilon\wedge(\theta-t)}q(t+u)d\ell_u^0(X)\right\} \nonumber \\
  &= \lim_{\epsilon\downarrow 0}\left\{\frac{1}{\epsilon}\,\E\int_{0}^{\epsilon\wedge(\theta-t)}q(t+u)\delta_0(X_{t+u})du\right\} \nonumber \\
  &= q(t)\delta_0(x), \nonumber
\end{align}
and the claim is proved.
\end{proof}

\vspace{4pt}

6. To solve the optimal stopping problem in \eqref{eqn:prob} we embed it into a Markovian framework where the process $X$ starts at time $t$ with value $x$.
Trivially, if $t\geq\theta$ then the optimal stopping problem has a zero value.
Therefore, in what follows we assume that $t<\theta$ and we formulate the embedded problem as
\begin{equation}\label{eqn:prob2}
  V(t,x)=\sup_{0\leq\tau\leq T-t}\mathbb{E}\big[X_{t+\tau}^{t,x}\textbf{1}_{\{t+\tau<\theta\}}\big]
\end{equation}
where the process $X=X^{t,x}$ is defined by
\begin{eqnarray}
  \left\{\begin{array}{ll}\label{eqn:Xembed}
  dX_{t+s}=-X_{t+s}f(t+s,X_{t+s})ds+d\bar{B}_{t+s}, & 0\leq s<\theta-t,\\
  X_t=x, & x\in\mathbb{R},\\
\end{array}\right.
\end{eqnarray}
and where $X$ is killed elastically at zero according to \eqref{eqn:kill}.
Recall that $T$ denotes the upper limit of the support of $\mu$ and hence the set of admissible stopping times in \eqref{eqn:prob2} is fixed (i.e.,~independent of $\theta$).
However, the indicator function ensures that if stopping occurs after $\theta$ then the payoff is zero.

\vspace{4pt}

\begin{remark}
Noting the killing rate derived in \eqref{eqn:kill}, and comparing \cite[][Section III.18]{RW}, the objective function in \eqref{eqn:prob2} can alternatively be expressed as
\begin{equation}\label{eqn:local}
  \E\big[X_{t+\tau}^{t,x}\textbf{1}_{\{t+\tau<\theta\}}\big]=\E\Big[X_{t+\tau}^{t,x}\E\big[\textbf{1}_{\{t+\tau<\theta\}}\,|\,\mathcal{F}_{t+\tau}^X\big]\Big]=\E\Big[e^{-\int_{0}^{\tau}q(t+u)d\ell_u^0(X)}X_{t+\tau}^{t,x}\Big].
\end{equation}
This formulation reveals that the value received upon stopping ($X$) is effectively discounted over time at a rate proportional to the local time of $X$ spent at zero.
The appearance of local time in this way also indicates that we should expect a jump in the $x$-derivative of the value function across zero (cf. \cite[][p.~123]{BS}).
However, we will delay further discussion of this feature to the specific examples considered in Sections \ref{sec:gamma} and \ref{sec:beta}.
\end{remark}

\vspace{4pt}

7. From \eqref{eqn:prob2} it is evident that $V(t,x)\geq x$ for all $(t,x)\in[0,\infty)\times\R$.
As such, we define the continuation region $\C:=\{(t,x)\in[0,\infty)\times\R\,|\,V(t,x)>x\}$ and the stopping region $\D:=\{(t,x)\in[0,\infty)\times\R\,|\,V(t,x)=x\}$.
These regions are of importance in the general theory of optimal stopping (see \cite{PS}) and the structure of the optimal stopping region $\D$ will in general depend on the chosen prior $\mu$.
However, some general properties of the optimal stopping and continuation regions are discussed below.

From \eqref{eqn:Xembed} and an application of the optional sampling theorem we have that for any given stopping time $\tau$
\begin{equation}\label{eqn:nostopneg}
  \E\big[X_{t+\tau}^{t,x}\textbf{1}_{\{t+\tau<\theta\}}\big]=x-\E\int_0^{\tau\wedge(\theta-t)}X_{t+s}^{t,x}f(t+s,X_{t+s}^{t,x})ds.
\end{equation}
We note from \eqref{eqn:nostopneg} that, since $f\geq 0$, it will not be optimal to stop when $x$ is negative, hence $\{(t,x)\in[0,\infty)\times\R\,|\,x<0\}\in\C$.
In other words, since it is known that the process will eventually pin (yielding a payoff of zero), it would not be optimal to stop and receive a negative payoff before this time.
Given this fact, it is therefore evident that if a single optimal stopping boundary were to exist it could not be of the form $\C=\{(t,x)\in[0,\infty)\times\R\,|\,x>b(t)\}$, i.e.~a single lower boundary.

In the specific examples considered in Sections \ref{sec:gamma} and \ref{sec:beta} we find (and verify) that the optimal stopping strategy takes on the form of a one-sided \emph{upper} boundary.
However, informal (numerical) investigation of other priors indicate the possible existence of multiple stopping boundaries and disjoint continuation regions.
One would ideally like to establish conditions under which a one-sided stopping region can be expected.
However, such results appear difficult to obtain in the present setting, since standard methods are complicated by the dual effect of the prior on the conditional dynamics and the random time horizon of the optimal stopping problem.
Such structural questions about the optimal stopping region must therefore be tackled on a case-by-case basis.

\vspace{4pt}


8. To close this section we briefly review the solution to the classical Brownian bridge problem with a known pinning time $T>0$ which will be used in our subsequent analysis.
When $T$ is known and fixed, the stopping problem \eqref{eqn:prob2} has an explicit solution (first derived in \cite{S} and later in \cite{EW}) given by
\begin{equation}\label{eqn:certpinV}
  V^T(t,x)=\left\{\begin{array}{ll}
  \sqrt{2\pi(T-t)}(1-B^2)e^{\frac{x^2}{2(T-t)}}\Phi\left(\frac{x}{\sqrt{T-t}}\right), & x<b^T(t),\\
  x, & x\geq b^T(t),
\end{array}\right.
\end{equation}
for $t<T$ and $V^T(1,0)=0$.
The function $\Phi(y)$ denotes the standard cumulative normal distribution function and $b^T(t):=B\sqrt{T-t}$ with $B$ the unique positive solution to
\begin{equation}\label{eqn:classicB}
  \sqrt{2\pi}(1-B^2)e^{\frac{1}{2}B^2}\Phi(B)=B,
\end{equation}
which is approximately 0.839924.
Further, the optimal stopping strategy is given by $\tau^*=\inf\{s\geq 0\,|\,X_{t+s}\geq b^T(t+s)\}$ for all $t<T$ and hence the optimal stopping region is given by
\begin{equation}\label{eqn:stoppingD}
  \D^T:=\{(t,x)\in[0,T]\times\R\,|\,x\geq b^T(t)\}
\end{equation}
with $\C^T:=([0,T]\times\R)\,\backslash\,\D^T$ denoting the continuation region.

\vspace{4pt}

\begin{remark}
It is intuitive that if $\mathrm{supp}\,\mu\subseteq[0,T]$ in \eqref{eqn:prob2}, then $V\leq V^T$, where $V^T$ is given in \eqref{eqn:certpinV}, and consequently $\D^T\subseteq\D$, where $\D^T$ is given in \eqref{eqn:stoppingD}. Formally, this can be seen by considering the value function in \eqref{eqn:prob2} if the true value of $\theta$ was revealed to the stopper immediately (at $t=0+$). Denoting this value by $\overline{V}$ it is clear that $\overline{V}=\int_t^T V^r\mu(dr)\leq V^T$ due to $\mathrm{supp}\,\mu\subseteq[0,T]$. Furthermore, since the set of stopping times when knowing the pinning time is larger than when not knowing the pinning time it is clear that $V\leq\overline{V}$ and the stated inequality follows.
\end{remark}

\section{The case of a gamma distributed prior}\label{sec:gamma}
1. It is perhaps most obvious to consider an exponentially distributed prior for $\theta$, however it appears that explicit computation of the function $f$ in \eqref{eqn:f} for such distributions is not possible.
A related distribution for which $f$ can be computed explicitly however is a gamma distribution $\Gamma(\alpha,\beta)$ when $\alpha=n-1/2$ for positive integers $n$.
Note that this distribution is supported on the semi-infinite interval $[0,\infty)$, but that the pinning time is still integrable with $\E[\theta]=\alpha/\beta$.
We also note that when $\beta=2$ the gamma distribution with $\alpha=\nu/2$ reduces to a chi-squared distribution of $\nu$ degrees of freedom, i.e. $\Gamma(\nu/2,2)=\chi^2_\nu$.
Therefore, this case encompasses chi-squared distributions with odd degrees of freedom, i.e.~$\chi^2_{2n-1}$ for $n\in\mathbb{Z}^+$.
For these distributions we have the following result.

\begin{proposition}\label{prop:gamf}
Let $\kappa=0$ and $\theta\sim\Gamma(\alpha,\beta)$ with $\beta>0$ and $\alpha=n-1/2$ (for $n\in\mathbb{Z}^+$) such that
\begin{equation}\label{eqn:gammadens}
  \mu(dr)=\frac{\beta^{\alpha}}{\Gamma(\alpha)}r^{\alpha-1}e^{-\beta r}dr.
\end{equation}
The function $f$ in \eqref{eqn:f} can be calculated explicitly as
\begin{equation}\label{eqn:gamf}
  f(t,x)=\frac{\sqrt{2\beta}}{|x|}Q(t,x)
\end{equation}
with
\begin{equation}\label{eqn:gamf2}
  Q(t,x):=\frac{\sum_{k=0}^{n-1}\binom{n-1}{k}t^k\left(2\beta/x^2\right)^{k/2}K_{n-k-3/2}\left(\sqrt{2\beta}|x|\right)}{\sum_{k=0}^{n-1}\binom{n-1}{k}t^k\left(2\beta/x^2\right)^{k/2}K_{n-k-1/2}\left(\sqrt{2\beta}|x|\right)},
\end{equation}
where $K_\nu(\cdot)$ is the modified Bessel function of the second kind (of order $\nu$).
Hence the drift function in \eqref{eqn:X} is given by
\begin{equation}\label{eqn:gamdrift}
  -xf(t,x)=-\sqrt{2\beta}\,\mathrm{sgn}(x)Q(t,x)
\end{equation}
where \emph{sgn} is defined as
\begin{eqnarray}
  \mathrm{sgn}(x):=\left\{\begin{array}{rl}
  -1,& \textrm{ for }\, x<0\\
  0,& \textrm{ for }\, x=0\\
  +1,& \textrm{ for }\, x>0.\\
  \end{array}\right.
\end{eqnarray}
\end{proposition}

\begin{proof}
In order to compute \eqref{eqn:f} with the density \eqref{eqn:gammadens} we must evaluate the integral
  \[\frac{\beta^{\alpha}}{\Gamma(\alpha)}\int_t^{\infty}\frac{r^{\alpha-1}}{(r-t)^a}\sqrt{\frac{r}{r-t}}\exp\left(-\frac{rx^2}{2t(r-t)}-\beta r\right)dr\]
with $a=0$ corresponding to the integral in the denominator of \eqref{eqn:f} and $a=1$ to the integral in the numerator.
Letting $u=1/(r-t)$, the integral above reduces to
\begin{equation}\label{eqn:ident2}
  \frac{\beta^{\alpha}}{\Gamma(\alpha)}e^{-\tfrac{1}{2}x^2-\beta t}\int_0^{\infty}(1+tu)^{\alpha-1/2}u^{a-\alpha-1}e^{-\tfrac{1}{2}x^2 u-\beta/u}du.
\end{equation}
We were not able to find an explicit computation of the above integral for arbitrary $\alpha$.
However, it can be seen that if $\alpha-1/2$ is a non-negative integer then the term $(1+tu)^{\alpha-1/2}$ can be expanded into integer powers of $u$ and we can apply the following known integral identity (cf. \cite[p.~313]{Er1})
\begin{equation}\label{eqn:intident}
  \int_{0}^{\infty}u^{\nu-1}e^{-Au-B/u}du=2(B/A)^{\frac{\nu}{2}}K_{\nu}(2\sqrt{AB})
\end{equation}
valid for $A,B>0$.
Letting $\alpha=n-1/2$ we can identify $\nu=a+k-n+1/2$, $A=x^2/2$, $B=\beta$, and using $a=0$ and $1$, respectively, to perform the integration we obtain the stated expression.
\end{proof}

\begin{corollary}\label{cor:gamf}
When $\theta\sim \Gamma(1/2,\beta)$ for $\beta>0$ the function $f$ in \eqref{eqn:gamf} becomes time independent and given by
\begin{equation}\label{eqn:gamfhalf}
  f(t,x)=\sqrt{2\beta}/|x|.
\end{equation}
\end{corollary}

\begin{proof}
Setting $n=1$ in \eqref{eqn:gamf2} reveals that
\begin{equation}
  Q(t,x)=\frac{K_{-1/2}(\sqrt{2\beta}|x|)}{K_{1/2}(\sqrt{2\beta}|x|)}=1
\end{equation}
upon noting that $K_\nu=K_{-\nu}$, which produces the desired result.
\end{proof}

\begin{remark}
Corollary \ref{cor:gamf} reveals the remarkable property that when $\theta\sim\Gamma(1/2,\beta)$ the dynamics in \eqref{eqn:X} are \emph{time homogeneous} and dependent only on the sign of $X$.
As such, a movement of the process away from zero increases the stopper's expected pinning time, and hence decreases the expected pinning force (via the $1/(\theta-t)$ term), by just enough to offset the increased pinning force due to the process being further way from zero (via the $-X_t$ term).
\end{remark}


\vspace{4pt}

Next, we observe from \eqref{eqn:gamf2} that $Q(t,x)\leq 1$ (since $\nu\mapsto K_\nu$ is increasing for $\nu\geq 0$) and hence we conclude from \eqref{eqn:gamf} that $\lim_{x\to 0}f(t,x)=\infty$, as expected from the strictly positive density of $\mu$ at $r=t$ for all $t$.
We also observe that the drift function in \eqref{eqn:gamdrift} has a discontinuity at $x=0$ since $\lim_{x\downarrow 0}Q(t,x)>0$.
However, despite this discontinuity, the SDE in \eqref{eqn:X} has a unique strong solution since $|xf(t,x)|\leq\sqrt{2\beta}$ and the drift is bounded (cf. \cite{Z}).
Finally, we note that in the case when $\alpha=1/2$, the SDE in \eqref{eqn:X} is often referred to as \emph{bang-bang Brownian motion} or \emph{Brownian motion with alternating drift}.
This process has arisen previously in the literature in the study of reflected Brownian motion with drift.
In fact, the drawdown of a Brownian motion with drift (i.e., the difference between the current value and its running maximum) is equal in law to the absolute value of bang-bang Brownian motion (see \cite{GS}). 
We also refer the reader to the recent work of \cite{MS} who consider discounted optimal stopping problems for a related process with a discontinuous (broken) drift.

\vspace{4pt}
%
%
%
%

2. Turning now to the optimal stopping problem in \eqref{eqn:prob2}, we find that the time homogeneity in the case when $\theta\sim \Gamma(1/2,\beta)$ allows us to solve the problem in closed form.
For all other values of $\alpha$ considered in Proposition \ref{prop:gamf} the problem is time \emph{in}homogeneous and must be solved numerically.
We therefore restrict our attention to the $\alpha=1/2$ ($n=1$) case and expose the solution there in full detail.
The other cases are left for the subject of future research.

To derive our candidate solution to \eqref{eqn:prob2} when $\theta\sim\Gamma(1/2,\beta)$ we note that $\mu$ has a continuous density, and computing the killing rate from \eqref{eqn:kill}, we see that $q(t)=\sqrt{2\beta}$, a constant.
Therefore the optimal stopping problem also becomes time homogeneous.
To proceed, we assume the optimal stopping strategy is of the form $\tau=\inf\{s\geq 0\,|\,X_{t+s}^{t,x}\geq b\}$ for some constant $b$ to be determined.
This assumption with be justified in the verification arguments below (Theorem \ref{thm:verifgamma}).
Under this form of stopping strategy the general theory of optimal stopping (see, for example, \cite{PS}) indicates that the value function and optimal stopping boundary should satisfy the following free-boundary problem.
\begin{eqnarray}
  \left\{\begin{array}{rl}\label{eqn:FBPgamma}
  \mathbb{L}_X\widehat{V}(x)=0, & \textrm{ for } x\in(-\infty,0)\cup(0,b),\\
  \widehat{V}(x)=x, & \textrm{ for } x\geq b,\\
  \widehat{V}^{\prime}(x)=1, & \textrm{ at } x=b,\\
  \widehat{V}(0-)=\widehat{V}(0+), & \textrm{ at } x=0,\\
  \widehat{V}^{\prime}(0+)-\widehat{V}^{\prime}(0-)=2\sqrt{2\beta}\,\widehat{V}(0), & \textrm{ at } x=0,\\
  \widehat{V}(x)<\infty, & \textrm{ as } x\to-\infty,\\
  \end{array}\right.
\end{eqnarray}
where $\mathbb{L}_X$ denotes the infinitesimal generator of $X$.
Recall that the derivative condition at $x=b$ represents smooth pasting and the derivative condition at $x=0$ is due to the elastic killing of the process at zero (cf. \cite[][p.~29]{BS}).
Problem \eqref{eqn:FBPgamma} can be solved explicitly to yield the following candidate for the optimal stopping value
\begin{eqnarray}
  \widehat{V}(x)=\left\{\begin{array}{ll}\label{eqn:gammaval}
  be^{-1}, & \textrm{ for } x\leq 0,\\
  be^{(x/b)-1}, & \textrm{ for } 0<x<b,\\
  x, & \textrm{ for } x\geq b,
  \end{array}\right.
\end{eqnarray}
where the optimal stopping threshold is given by $b=1/2\sqrt{2\beta}$.

Figure \ref{fig:gammaval} plots the value function in \eqref{eqn:gammaval} and the associated optimal stopping boundary for various values of $\beta$.
Note the kink in the value function at zero.
We can also see that $\beta\mapsto b(\beta)$ is decreasing and hence $\beta\mapsto\widehat{V}$ is also decreasing.
This is consistent with the fact that $\E[\theta]=1/2\beta$ and hence as $\beta$ increases the process is expected to pin sooner and the option value to stop smaller.
We also observe that $\lim_{\beta\downarrow 0}b(\beta)=\infty$ which is consistent with the fact that it would never be optimal to stop in this limit as the process would become a standard Brownian motion (which never pins).

\begin{figure}[htp!]\centering
  \psfrag{x}{$x$}
  \psfrag{V}{$\widehat{V}$}
  \psfrag{beta}{$\beta$}
  \psfrag{b}{$b$}
  \includegraphics[width=0.45\textwidth]{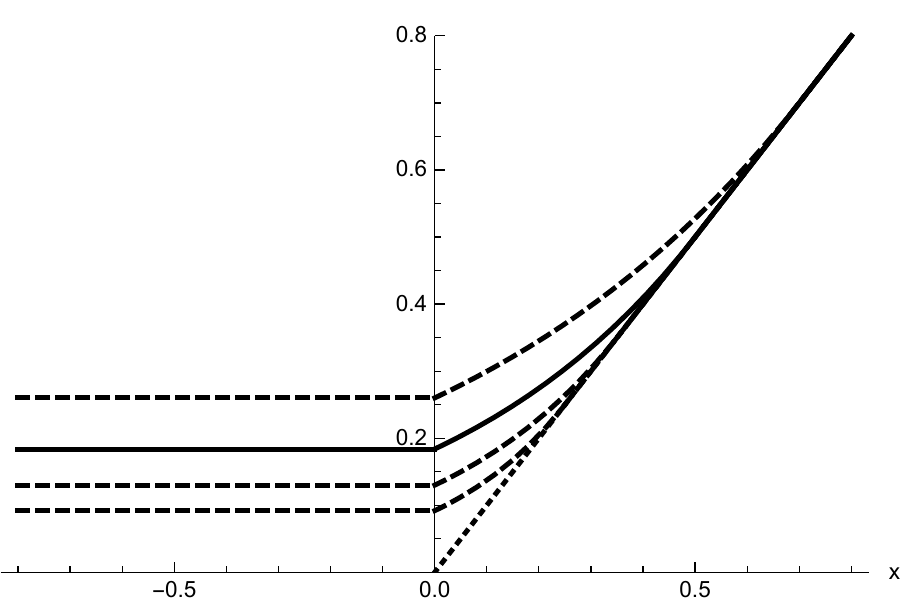}\hspace{2pt}
  \includegraphics[width=0.45\textwidth]{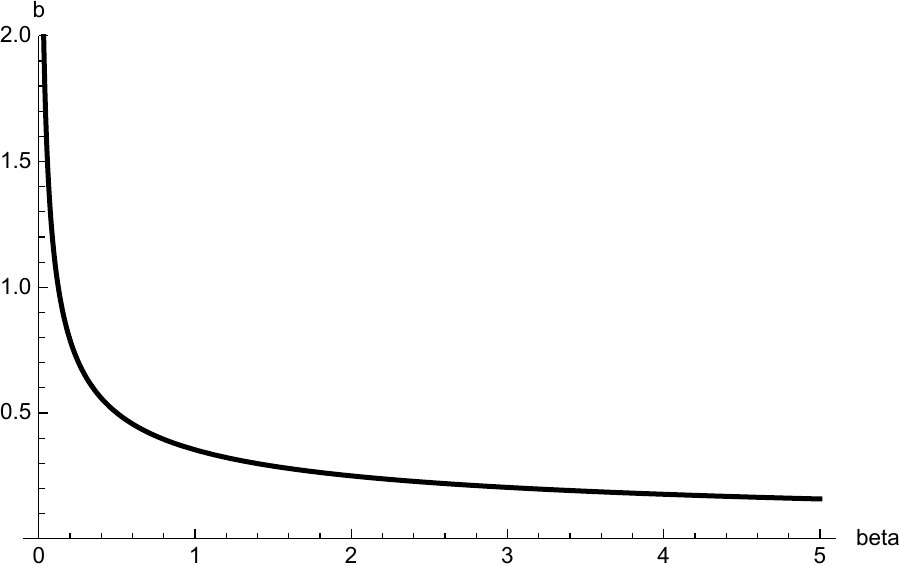}
  \caption{On the left: The candidate value function $\widehat{V}$ for the case of $\theta\sim\Gamma(1/2,\beta)$ given by \eqref{eqn:gammaval} for $\beta=1/2$ (solid line) and for $\beta=\{1/4,1,2\}$ (dashed lines). Lower curves correspond to higher values of $\beta$. On the right: The optimal stopping boundary $b=1/2\sqrt{2\beta}$.} \label{fig:gammaval}
\end{figure}

\vspace{4pt}

3. We conclude this section with the verification that the candidate value function in \eqref{eqn:gammaval} is indeed the solution to the optimal stopping problem.

\begin{theorem}\label{thm:verifgamma} \emph{(Verification).}
The value function $\widehat{V}$ defined in \eqref{eqn:gammaval} coincides with the function $V$ defined in \eqref{eqn:prob2} with $\theta\sim\Gamma(1/2,\beta)$.
Moreover, the stopping time $\tau^*=\inf\{s\geq 0\,|\,X_{t+s}^{t,x}\geq 1/2\sqrt{2\beta}\}$ is optimal.
\end{theorem}

\begin{proof}
Setting $X=X^{t,x}$ to simplify the notation, we first note that the problem is time homogeneous and hence for an arbitrary stopping time $\tau$,
\begin{align}
  \E\big[\widehat{V}(X_{\tau\wedge\theta})\big] &= \E\big[\widehat{V}(X_{\tau})\mathbf{1}_{\{\tau<\theta\}}\big] + \E\big[\widehat{V}(X_{\theta})\mathbf{1}_{\{\theta\leq\tau\}}\big] \nonumber \\
  &= \E\big[\widehat{V}(X_{\tau})\mathbf{1}_{\{\tau<\theta\}}\big] + \widehat{V}(0)\E\big[\mathbf{1}_{\{\theta\leq\tau\}}\big] \nonumber \\
  &\geq \E\big[X_{\tau}\mathbf{1}_{\{\tau<\theta\}}\big] + \widehat{V}(0)\P\left(\theta\leq\tau\right) \nonumber \\
  &= \E\big[X_{\tau}\mathbf{1}_{\{\tau<\theta\}}\big] + \sqrt{2\beta}\,\widehat{V}(0)\E[\ell^0_{\tau\wedge\theta}(X)], \label{eqn:ineqgam}
\end{align}
where the last equality above is due to \eqref{eqn:probpin} upon setting $t=0$ and $\epsilon=\tau$ and recalling that $q(t)=\sqrt{2\beta}$.
Secondly, an application of the local time-space formula (cf. \cite{P}), given that $\widehat{V}^{\prime}$ is continuous across $x=b$ but not across $x=0$, yields
\begin{align}
  \E\big[\widehat{V}(X_{\tau\wedge\theta})\big] &= \widehat{V}(x)+\E\int_0^{\tau\wedge\theta}\mathbb{L}_X\widehat{V}(X_{s})\mathbf{1}_{\{X_{s}\neq 0\textrm{ or }b\}}ds\nonumber \\
  &\quad\quad +\E\left[M_{\tau\wedge\theta}\right]+\E\int_0^{\tau\wedge\theta}\frac{1}{2}\big(\widehat{V}^{\prime}(X_{s}+)-\widehat{V}^{\prime}(X_{s}-)\big)\mathbf{1}_{\{X_{s}=0\}}d\ell_s^0(X) \nonumber \\
  &= \widehat{V}(x)-\sqrt{2\beta}\,\E\int_0^{\tau\wedge\theta}\textrm{sgn}(X_{s})\mathbf{1}_{\{X_{s}>b\}}ds \nonumber \\
  &\quad\quad +\E\left[M_{\tau\wedge\theta}\right]+\sqrt{2\beta}\,\E\int_0^{\tau\wedge\theta}\widehat{V}(X_s)\mathbf{1}_{\{X_{s}=0\}}d\ell_s^0(X) \nonumber \\
  &=: \widehat{V}(x)+\E\left[\Lambda_{\tau\wedge\theta}\right]+\E\left[M_{\tau\wedge\theta}\right]+\sqrt{2\beta}\,\widehat{V}(0)\E\int_0^{\tau\wedge\theta}d\ell_s^0(X)\nonumber \\
  &= \widehat{V}(x)+\E\left[\Lambda_{\tau\wedge\theta}\right]+\E\left[M_{\tau\wedge\theta}\right]+\sqrt{2\beta}\,\widehat{V}(0)\E[\ell_{\tau\wedge\theta}^0(X)] \nonumber 
\end{align}
where $M_t:=\int_0^{t}\widehat{V}^{\prime}(X_{s})\mathbf{1}_{\{X_{s}\neq 0\textrm{ or }b\}}d\bar{B}_{s}$ is a local martingale and $\Lambda$ is a decreasing process since $b\geq 0$.
Thirdly, combining \eqref{eqn:ineqgam} with the above equality we see that
\begin{equation}\label{eqn:verifgam}
  \E\left[X_{\tau}\mathbf{1}_{\{\tau<\theta\}}\right]\leq\E\big[\widehat{V}(X_{\tau})\mathbf{1}_{\{\tau<\theta\}}\big]=\widehat{V}(x)+\E\left[\Lambda_{\tau\wedge\theta}\right]+\E\left[M_{\tau\wedge\theta}\right]\leq \widehat{V}(x)
\end{equation}
where the last inequality follows from the optional sampling theorem upon noting that $\widehat{V}^{\prime}$ is bounded and hence $M$ is a true martingale.
Consequently, taking the supremum over all admissible stopping times in \eqref{eqn:verifgam} yields
  \[V(x)=\sup_{\tau\geq 0}\E[X_{\tau}\mathbf{1}_{\{\tau<\theta\}}]\leq\widehat{V}(x).\]
To establish the reverse inequality note that, since $\widehat{V}(X_{\tau^*})=X_{\tau^*}$ and $\E[\Lambda_{\tau^*\wedge\theta}]=0$, both inequalities in \eqref{eqn:verifgam} become equalities for $\tau=\tau^*$.
Thus
  \[V(x)\geq\E[X_{\tau^*}\mathbf{1}_{\{\tau^*<\theta\}}]=\widehat{V}(x),\]
completing the proof.
\end{proof}

\section{The case of a beta distributed prior}\label{sec:beta}
1. Another natural prior to consider is that of a beta distribution $B(\alpha,\beta)$ for $\alpha,\beta>0$.
Such a distribution allows for the pinning time to occur on a bounded interval which is often the case in real life applications of Brownian bridges.
Once more, while explicit computation of the function $f$ does not appear possible under a general beta distribution for arbitrary $\alpha$ and $\beta$, it is possible to obtain an explicit expression when $\alpha=1/2$ (for any $\beta>0$) and when $\kappa=0$.
Therefore, in this section we make the standing assumption that $\alpha=1/2$ and $\kappa=0$. 
The case $\alpha=\beta=1/2$ is also of particular interest since $B(1/2,1/2)$ corresponds to the well-known \emph{arcsine} distribution (see \cite{GJL}).
We further note that, without loss of generality, the beta distribution defined over $[0,1]$ can be taken (rather than over $[0,T]$) since the scaling $t\to t/T$ and $X\to X/\sqrt{T}$ could be used otherwise.
Finally to aid with the interpretation of our results, we recall that the unconditional expectation of $\theta\sim B(1/2,\beta)$ can be calculated as $\E[\theta]=1/(1+2\beta)$ and hence the expected pinning time is decreasing in $\beta$ with $\lim_{\beta\downarrow 0}\E[\theta]=1$ and $\lim_{\beta\uparrow\infty}\E[\theta]=0$.

\begin{proposition}\label{prop:betaf}
Let $\kappa=0$ and $\theta\sim B(1/2,\beta)$ for $\beta>0$, such that
\begin{equation}\label{eqn:betadense}
  \mu(dr)=\frac{(1-r)^{\beta-1}}{\sqrt{r}\,B(1/2,\beta)}dr.
\end{equation}
The function $f$ in \eqref{eqn:f} can thus be computed explicitly as
\begin{equation}\label{eqn:fbeta}
  f(t,x)=\frac{g\left(x/\sqrt{1-t}\right)}{1-t} \quad \textrm{with} \quad g(z):=\frac{U\left(\beta,\tfrac{3}{2},\tfrac{1}{2}z^{2}\right)}{U\left(\beta,\tfrac{1}{2},\tfrac{1}{2}z^{2}\right)}
\end{equation}
where $U$ denotes Tricomi's confluent hypergeometric function (cf. \cite{T}).
\end{proposition}

\begin{proof}
To compute $f$ under this distribution we must evaluate the integral
  \[\frac{1}{B(1/2,\beta)}\int_{t}^{1}(1-r)^{\beta-1}(r-t)^{-a-\tfrac{1}{2}}\exp\left(-\frac{rx^2}{2t(r-t)}\right)dr\]
with $a=0$ corresponding to the integral in the denominator of \eqref{eqn:f} and $a=1$ to the integral in the numerator.
Letting $u=(1-r)/(r-t)$, the integral above reduces to
\begin{equation}\label{eqn:int1}
  \frac{(1-t)^{\beta-a-\tfrac{1}{2}}}{B(\nicefrac{1}{2},\beta)}e^{-\frac{x^2}{2t(1-t)}}\int_{0}^{\infty}u^{\beta-1}(1+u)^{a-\beta-\tfrac{1}{2}}e^{-\frac{x^2 u}{2(1-t)}}du,
\end{equation}
which can be computed explicitly by noting the following integral representation of Tricomi's confluent hypergeometric function $U$ (see \cite[p.~505]{AS})
  \[U(p,q,y)=\frac{1}{\Gamma(p)}\int_{0}^{\infty}u^{p-1}(1+u)^{q-p-1}e^{-yu}du\]
valid for $p,y>0$.
Identifying $p=\beta$, $q=a+1/2$ and $y=x^2/2(1-t)$ we thus have
\begin{equation}\label{eqn:int2}
  \int_{0}^{\infty}u^{\beta-1}(1+u)^{a-\beta-\tfrac{1}{2}}e^{-\frac{x^2 u}{2(1-t)}}du=\Gamma(\beta)U\left(\beta,a+\tfrac{1}{2},\tfrac{x^2}{2(1-t)}\right).
\end{equation}
Using \eqref{eqn:int2} in \eqref{eqn:int1} and substitution of \eqref{eqn:int1} into \eqref{eqn:f} (upon setting $a=0$ for the denominator and $a=1$ for the numerator) yields the desired result.
\end{proof}


\begin{corollary}\label{cor:betaf}
When $\theta\sim B(1/2,1/2)$ the function $g$ in \eqref{eqn:fbeta} corresponds to
\begin{equation}\label{eqn:garcsin}
  g(z)=\frac{1}{|z|\sqrt{2\pi}}\frac{e^{-\frac{1}{2}z^2}}{1-\Phi\left(|z|\right)}
\end{equation}
where $\Phi$ is the cumulative normal distribution function.
\end{corollary}

\begin{proof}
The expression can be obtained directly from \eqref{eqn:fbeta} after noting from \cite[][p.~510]{AS} that $U\left(\tfrac{1}{2},\tfrac{3}{2},y\right)=1/\sqrt{y}$ and $U\left(\tfrac{1}{2},\tfrac{1}{2},y\right)=2\sqrt{\pi}e^y\left[1-\Phi(\sqrt{2y})\right]$.
\end{proof}


\vspace{4pt}

\begin{figure}[htp!]\centering
  \psfrag{y}{$x$}
  \psfrag{g}{$f$}
  \psfrag{beta}{$\beta\uparrow$}
  \includegraphics[width=0.48\textwidth]{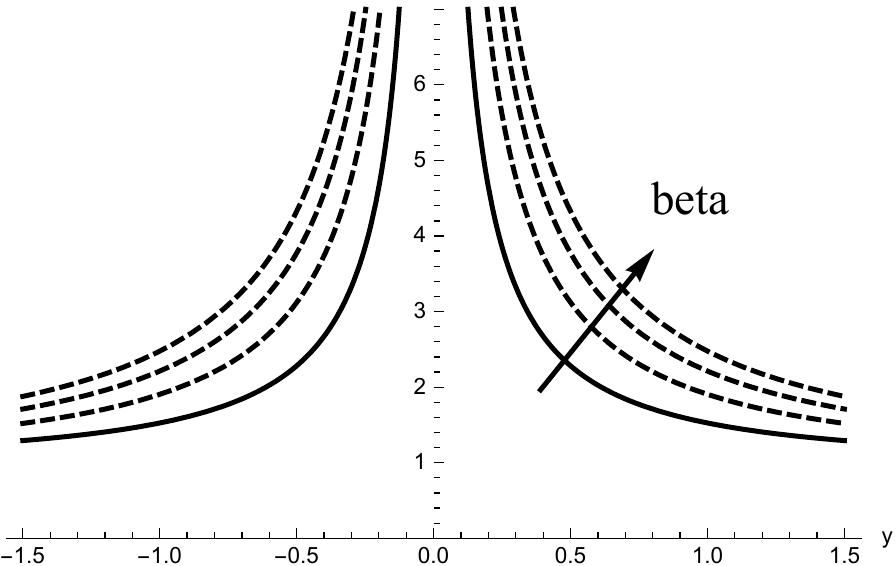}\hspace{2pt}
  \includegraphics[width=0.48\textwidth]{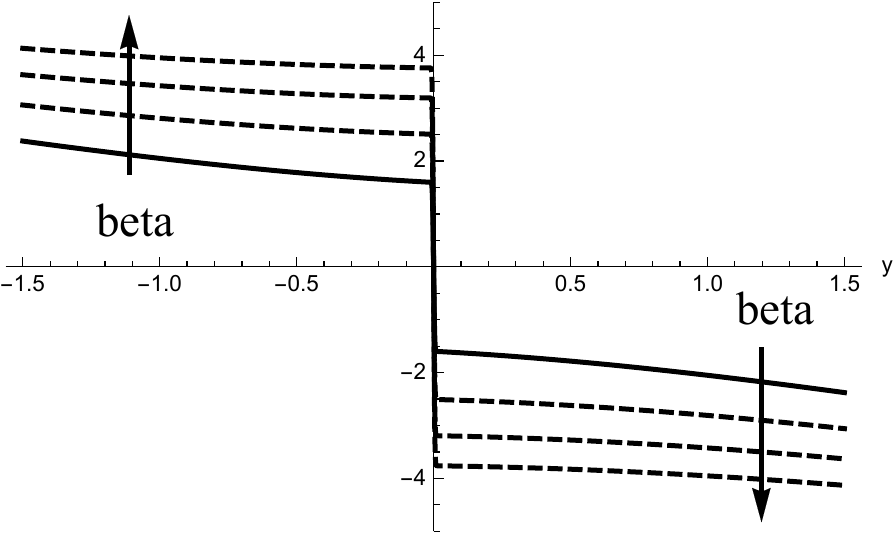}
  \caption{On the left: The function $f(0,x)$ plotted for $\beta=1/2$ (solid line) and for $\beta=\{1,3/2,2\}$ (dashed lines). On the right: The function $-xf(0,x)$ plotted for the same values of $\beta$.}\label{fig:gfunc}
\end{figure}

\vspace{4pt}

2. In Figure \ref{fig:gfunc} we plot the function $f$ given in \eqref{eqn:fbeta} for various values of $\beta$. We also plot the associated drift function $-xf(t,x)$.
We confirm the results of Proposition \ref{prop:fdecrease} that $f$ is indeed larger for values of $x$ closer to zero.
We also observe that $\lim_{x\to 0}f(t,x)=\infty$ which, as noted previously, is consistent with the fact that there is a strictly positive density of the beta distribution at $r=t$ for all $t\in(0,1)$.
Consequently, the drift function is seen, once again, to have a discontinuity at zero.

\vspace{4pt}

3. The optimal stopping problem under a beta distribution is clearly time inhomogeneous, however we are able to exploit an inherent symmetry in the problem to derive a candidate solution.
Specifically, we observe that the problem can be reduced to solving a one dimensional (time-homogeneous) boundary value problem.
Moreover, the optimal stopping strategy in \eqref{eqn:prob2} has a square-root form, i.e.~$\tau^*=\inf\{t\geq 0\,|\,X_{t}>A(\beta)\sqrt{1-t}\}$ for some constant $A(\beta)$.

Assuming a one-sided stopping region, the general theory of optimal stopping indicates that the value function and optimal stopping boundary should satisfy the following free-boundary problem, where $\partial_1$ denotes differentiation with respect to the first argument.
\begin{eqnarray}\label{eqn:FBP2}
  \left\{\begin{array}{rl}
  \left(\partial_1+\mathbb{L}_X\right)\widehat{V}(t,x)=0, & \textrm{ for } x\in(-\infty,0)\cup(0,b(t)),\\
  \widehat{V}(t,x)=x, & \textrm{ for } x\geq b(t),\\
  \widehat{V}_x(t,x)=1, & \textrm{ at } x=b(t),\\
  \widehat{V}(t,0+)=\widehat{V}(t,0-), & \textrm{ at } x=0,\\
  \widehat{V}_x(t,0+)-\widehat{V}_x(t,0-)=2q(t)\widehat{V}(t,0), & \textrm{ at } x=0,\\
  \widehat{V}(t,x)=0, & \textrm{ at } x=-\infty,\\
 \end{array}\right.
\end{eqnarray}
for $t<1$ and $\widehat{V}(1,0)=0$ (since pinning must happen at or before $t=1$).
Note that the final condition in \eqref{eqn:FBP2} can be justified from the knowledge that $V\leq V^T$ and $\lim_{x\to -\infty}V^T=0$ (for $T=1$).
We also note that the function $q$ defined in \eqref{eqn:kill} can be computed explicitly in this case to be $q(t)=\sqrt{\tfrac{2\pi}{1-t}}\tfrac{1}{B(1/2,\beta)}$.

Given \eqref{eqn:FBP2} and the form of $f$ found in \eqref{eqn:fbeta} we make the ansatz $b(t)=A\sqrt{1-t}$ and further that $\widehat{V}(t,x)=\sqrt{1-t}\,u(z)$ where $z=x/\sqrt{1-t}$.
Problem \eqref{eqn:FBP2} is thus transformed into
\begin{eqnarray}
  \left\{\begin{array}{rl}\label{eqn:FBP3}
  u^{\prime\prime}(z)+z\big(1-2g(z)\big)u^{\prime}(z)-u(z)=0, & \textrm{ for } z\in(-\infty,0)\cup(0,b),\\
  u(z)=z, & \textrm{ for } z\geq A,\\
  u^{\prime}(z)=1, & \textrm{ at } z=A,\\
  u(0+)=u(0-), & \textrm{ at } z=0,\\
  u^{\prime}(0+)-u^{\prime}(0-)=\frac{2\sqrt{2\pi}}{B(1/2,\beta)}u(0), & \textrm{ at } z=0,\\
  u(z)=0, & \textrm{ at } z=-\infty,\\
 \end{array}\right.
\end{eqnarray}
where $g$ is as defined in \eqref{eqn:fbeta}.

\vspace{4pt}

4. We are now able to construct a solution to the above free-boundary problem using the so-called \emph{fundamental solutions} of the ODE in \eqref{eqn:FBP3}.
Firstly, we consider the region $z>0$ and denote the fundamental solutions in this region by $\psi$ and $\varphi$.
It is well known that these functions are positive and that $\psi$ and $\varphi$ are increasing and decreasing, respectively (cf.~\cite{BS}).
Furthermore, since these functions are defined up to an arbitrary multiplicative constant we can set $\psi(0+)=\varphi(0+)=1$ to simplify our expressions.
We therefore have $u(z)=C\psi(z)+D\varphi(z)$ for $z\in(0,A)$, where $C$ and $D$ are constants to be determined via the two boundary conditions at $z=A$ in \eqref{eqn:FBP3}.
Applying these conditions yields
\begin{equation}\label{eqn:CD}
  C=\frac{\varphi(A)-A\varphi^{\prime}(A)}{\varphi(A)\psi^{\prime}(A)-\varphi^{\prime}(A)\psi(A)} \quad \textrm{and} \quad D=\frac{A\psi^{\prime}(A)-\psi(A)}{\varphi(A)\psi^{\prime}(A)-\varphi^{\prime}(A)\psi(A)}.
\end{equation}
Next, the solution for $z<0$ can be constructed in a similar fashion to give $u(z)=C_-\psi_-(z)+D_-\varphi_-(z)$ for $z\in(-\infty,0)$, where $C_-$ and $D_-$ are constants to be determined and $\psi_-$ and $\varphi_-$ are the fundamental solutions for $z<0$.
In fact, it can be seen from the ODE in \eqref{eqn:FBP3} that we must have $\psi_-(z)=\varphi(-z)$ and $\varphi_-(z)=\psi(-z)$ since $g(z)$ is an even function.
To satisfy the boundary condition as $z\to -\infty$ it is clear that $D_- =0$.
In addition, to maintain continuity of $u$ at $z=0$ we must also have $C_- =C+D$ (upon using $\psi(0+)=\varphi(0+)=1$).
To summarize, the solution to \eqref{eqn:FBP3} can be expressed as
\begin{equation}\label{eqn:solbeta}
  u(z)=\left\{\begin{array}{ll}
  (C+D)\varphi(-z), & z\leq 0\\
  C\psi(z)+D\varphi(z), & 0<z<A,\\
  z, & z\geq A,
\end{array}\right.
\end{equation}
where $C$ and $D$ are given by \eqref{eqn:CD}.
Finally, the derivative condition at $z=0$ is used to fix the value of $A$, yielding the following equation
\begin{equation}\label{eqn:Acond}
  C\psi^{\prime}(0+)+D\varphi^{\prime}(0+)+(C+D)\varphi^{\prime}(0+)=\frac{2\sqrt{2\pi}}{B(1/2,\beta)}(C+D).
\end{equation}
Recalling that $C$ and $D$ depend on $A$, finding a value of $A$ that satisfies \eqref{eqn:Acond} gives a solution to \eqref{eqn:FBP3} via \eqref{eqn:solbeta}.
In fact, the following result demonstrates that there is a unique value of $A\geq 0$ satisfying \eqref{eqn:Acond} and hence there is a unique solution to the free-boundary problem in \eqref{eqn:FBP3} with $A\geq 0$.

\begin{proposition}\label{lem:Aunique}
For a given $\beta>0$, there is a unique $A\geq 0$ satisfying equation \eqref{eqn:Acond}.
\end{proposition}

\begin{proof}
Using \eqref{eqn:CD} in \eqref{eqn:Acond} and rearranging gives
\begin{equation}\label{eqn:hcond}
  p(A)=\frac{\psi^{\prime}(0+)+\varphi^{\prime}(0+)-\alpha}{2\varphi^{\prime}(0+)-\alpha}=:K \,\,\textrm{ where }\,\,p(z):=\frac{z\psi^{\prime}(z)-\psi(z)}{z\varphi^{\prime}(z)-\varphi(z)}
\end{equation}
and where we have defined the constant $\alpha:=2\sqrt{2\pi}/B(1/2,\beta)$.
It is clear that $p(0+)=1$ and direct differentiation of $p(z)$, upon using the ODE in \eqref{eqn:FBP3}, gives
\begin{equation*}
  p^{\prime}(z)=\frac{2z^2g(z)\big(\varphi^{\prime}(z)\psi(z)-\varphi(z)\psi^{\prime}(z)\big)}{\big(z\varphi^{\prime}(z)-\varphi(z)\big)^2} \leq 0,
\end{equation*}
since $g>0$ and $\psi$ and $\varphi$ are increasing and decreasing, respectively.
In addition, it can be shown that $\lim_{z\to\infty}p(z)=-\infty$ since $+\infty$ is a natural boundary for $X$ (and hence $\lim_{z\to\infty}\psi^{\prime}(z)=\infty$, see \cite[][p.~19]{BS}).
These properties of $p$ imply that a unique positive solution to \eqref{eqn:hcond} exists iff $K\leq 1$.
From its definition in \eqref{eqn:hcond}, this can be seen to be true upon noting that $\alpha\geq 0$ and that $\psi$ and $\varphi$ are increasing and decreasing, respectively.
\end{proof}

Since the function $g$ defined \eqref{eqn:fbeta} is rather complicated, finding a closed form expression for the functions $\psi$ and $\varphi$ appears unlikely.
However, standard numerical (finite-difference) methods can easily be used to construct these functions and determine the solution to \eqref{eqn:FBP3}.

In Figure \ref{fig:betaAVal} we plot the dependence of the constant $A$, and the corresponding value function $\widehat{V}$, on the parameter $\beta$.
We observe that $\lim_{\beta\downarrow 0}A(\beta)=B$, the solution in the known pinning case and given by the solution to \eqref{eqn:classicB}.
Further, we have that $\beta\mapsto A(\beta)$ is decreasing and hence the stopper will stop sooner for a larger $\beta$.
The corresponding value of stopping is thus also lower for larger $\beta$.
This dependence appears intuitive upon recalling that the unconditional expected pinning time is also decreasing in $\beta$.

\begin{figure}[htp!]\centering
  \psfrag{A}{$A$}
  \psfrag{V}{$\widehat{V}$}
  \psfrag{beta}{$\beta$}
  \psfrag{x}{$x$}
  \includegraphics[width=0.48\textwidth]{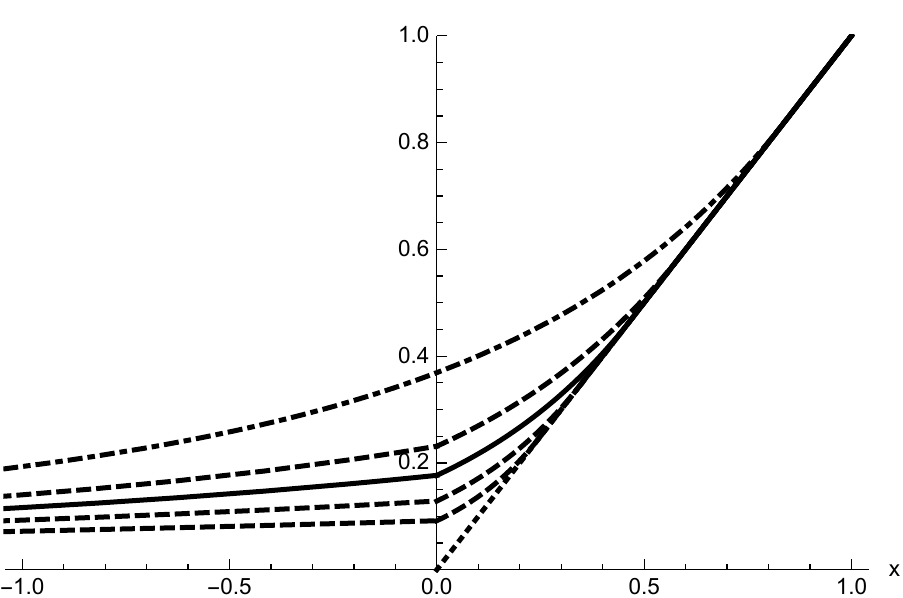}\hspace{2pt}
  \includegraphics[width=0.48\textwidth]{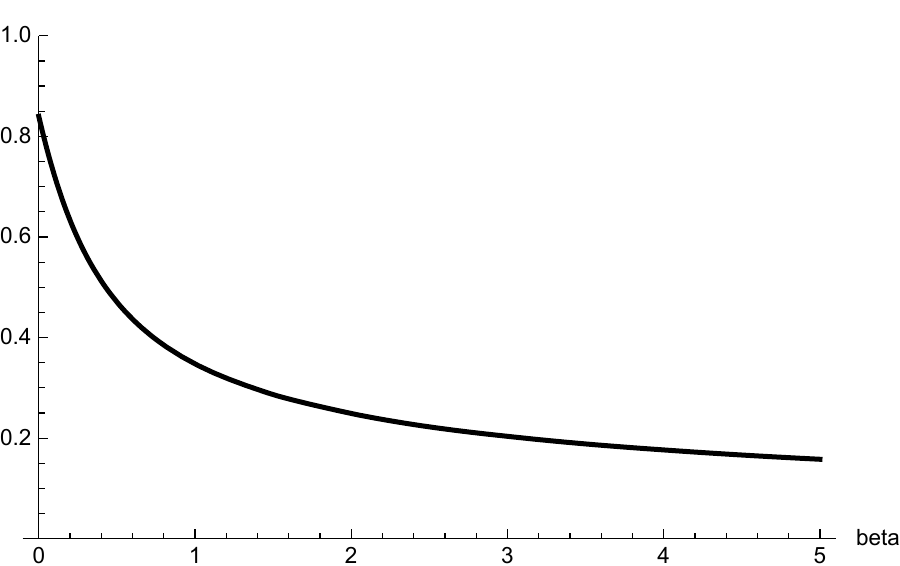}
  \caption{On the left: The value function $\widehat{V}$, determined numerically via \eqref{eqn:FBP3}, for the case of $\theta\sim B(1/2,\beta)$ for $\beta=1/2$ (solid line) and for $\beta=\{1/4,1,2\}$ (dashed lines). Lower lines correspond to higher values of $\beta$. The value function when the pinning time is known to equal $1$---given in \eqref{eqn:certpinV}---is also plotted for comparison (dot-dashed line). On the right: The unique constant $A$ for various values of $\beta$.} \label{fig:betaAVal}
\end{figure}

\vspace{4pt}

5. We conclude with our verification theorem, preceded by a necessary lemma.

\begin{lemma}\label{lem:compensator}
For any uniformly continuous function $h$ we have
\begin{equation}\label{eqn:Acomp}
  \E[h(\theta)\mathbf{1}_{\{\theta\leq t+\tau\}}\,|\,\mathcal{F}_t^X]=\E\int_0^{\tau\wedge(\theta-t)}h(t+u)q(t+u)d\ell_u^0(X)
\end{equation}
for any stopping time $\tau$.
\end{lemma}

\begin{proof}
Define the process $H_t:=h(\theta)\mathbf{1}_{\{\theta\leq t\}}$ for some uniformly continuous function $h$.
The aim is to show that the compensator of $H$ admits the representation
\begin{equation}\label{eqn:Hcomp}
  A_t=\int_0^{t\wedge\theta}h(s)q(s)d\ell_s^0(X),
\end{equation}
where $q$ is as defined in \eqref{eqn:kill}.
The proof of this representation follows analogous arguments to the proof of Theorem 3.2 in \cite{BBE2} and so full details are omitted in the interests of brevity.
The key difference however is that the process $A^\epsilon=(A^\epsilon_t)_{t\geq 0}$ defined as
\begin{equation}\label{eqn:hcomp}
  A_t^\epsilon:=\frac{1}{\epsilon}\int_0^t\left(h(\theta)\mathbf{1}_{\{s<\theta\}}-\E\left[h(\theta)\mathbf{1}_{\{s+\epsilon<\theta\}}\,|\,\mathcal{F}_s^X\right]\right)ds
\end{equation}
for every $\epsilon>0$ is used in place of $K^h$ in Eq.~(11) of \cite{BBE2}.
All arguments of the proof follow through without modification upon noting that the function $h$, like the density of $\theta$, is uniformly continuous.
Finally, using the compensator in \eqref{eqn:Hcomp} we identify that $\E[h(\theta)\mathbf{1}_{\{\theta\leq t+\tau\}}\,|\,\mathcal{F}_t^X]=\E[A_{t+\tau}-A_t]$, yielding the desired expression in \eqref{eqn:Acomp}.
\end{proof}

\begin{theorem}\label{thm:verifbeta} \emph{(Verification).}
The value function $\widehat{V}(t,x)=\sqrt{1-t}\,u\big(\tfrac{x}{\sqrt{1-t}}\big)$, where $u$ is the unique solution to \eqref{eqn:FBP3}, coincides with the function $V(t,x)$ defined in \eqref{eqn:prob2} with $\theta\sim B(1/2,\beta)$.
Moreover, the stopping time $\tau^*=\inf\{s\geq 0\,|\,X_{t+s}^{t,x}\geq b(t+s)\}$ is optimal, where $b(t)=A\sqrt{1-t}$ with $A\geq 0$ and uniquely determined via \eqref{eqn:Acond}.
\end{theorem}

\begin{proof}
Denoting $X=X^{t,x}$ for ease of notation, we first note that for an arbitrary stopping time $\tau$,
\begin{align}
  \E\big[\widehat{V}\big((t+\tau)\wedge\theta,X_{(t+\tau)\wedge\theta}\big)\big] &= \E\big[\widehat{V}(t+\tau,X_{t+\tau})\mathbf{1}_{\{t+\tau<\theta\}}\big] + \E\big[\widehat{V}(\theta,X_{\theta})\mathbf{1}_{\{\theta\leq t+\tau\}}\big] \nonumber \\
  &= \E\big[\widehat{V}(t+\tau,X_{t+\tau})\mathbf{1}_{\{t+\tau<\theta\}}\big] + \E\big[\widehat{V}(\theta,0)\mathbf{1}_{\{\theta\leq t+\tau\}}\big] \nonumber \\
  &\geq \E\big[X_{t+\tau}\mathbf{1}_{\{t+\tau<\theta\}}\big] + \E\big[\widehat{V}(\theta,0)\mathbf{1}_{\{\theta\leq t+\tau\}}\big] \nonumber \\
  &= \E\big[X_{t+\tau}\mathbf{1}_{\{t+\tau<\theta\}}\big] + \E\int_0^{\tau\wedge(\theta-t)}\widehat{V}(t+u,0)q(t+u)d\ell_u^0(X), \label{eqn:ineqbet}
\end{align}
where the last equality above is due to \eqref{eqn:Acomp} upon setting $h(s)=\widehat{V}(s,0)$ and noting that $\widehat{V}(s,0)=\sqrt{1-s}\,u(0)$ is uniformly continuous.
Secondly, an application of the local time-space formula (cf.~\cite{P}), upon noting that $\widehat{V}_t$ and $\widehat{V}_x$ are continuous across $x=b(t)$ but not across $x=0$, yields
\begin{align}
  &\E\big[\widehat{V}\big((t+\tau)\wedge\theta,X_{(t+\tau)\wedge\theta}\big)\big]=\widehat{V}(t,x)+\E\int_0^{\tau\wedge(\theta-t)}\left(\partial_1+\mathbb{L}_X\right)\widehat{V}(t+u,X_{t+u})\mathbf{1}_{\{X_{t+u}\neq 0\textrm{ or }b(t+u)\}}du\nonumber \\
  &\quad\quad\quad\quad+\E\left[M_{\tau\wedge(\theta-t)}\right]+\E\int_0^{\tau\wedge(\theta-t)}\frac{1}{2}\big(\widehat{V}_x(t+u,X_{t+u}+)-\widehat{V}_x(t+u,X_{t+u}-)\big)\mathbf{1}_{\{X_{t+u}=0\}}d\ell_u^0(X) \nonumber \\
  &\quad\quad\quad= \widehat{V}(t,x)-\E\int_0^{\tau\wedge(\theta-t)}X_{t+u}f(t+u,X_{t+u})\mathbf{1}_{\{X_{t+u}>b(t+u)\}}du \nonumber \\
  &\quad\quad\quad\quad+\E\left[M_{\tau\wedge(\theta-t)}\right]+\E\int_0^{\tau\wedge(\theta-t)}q(t+u)\widehat{V}(t+u,X_{t+u})\mathbf{1}_{\{X_{t+u}=0\}}d\ell_u^0(X) \nonumber \\
  &\quad\quad\quad=: \widehat{V}(t,x)+\E\left[\Lambda_{\tau\wedge(\theta-t)}\right]+\E\left[M_{\tau\wedge(\theta-t)}\right]+\E\int_0^{\tau\wedge(\theta-t)}q(t+u)\widehat{V}(t+u,0)d\ell_u^0(X) \nonumber
\end{align}
where $M_s:=\int_0^{s}\widehat{V}_x(t+u,X_{t+u})\mathbf{1}_{\{X_{t+u}\neq 0\textrm{ or }b(t+u)\}}d\bar{B}_{u}$ is a local martingale and $\Lambda$ is a decreasing process since $f\geq 0$ and $A\geq 0$ (and hence $b\geq 0$).
Thirdly, combining \eqref{eqn:ineqbet} with the above equality we see that
\begin{align}
  \E\left[X_{t+\tau}\mathbf{1}_{\{t+\tau<\theta\}}\right] &\leq \E\big[\widehat{V}(t+\tau,X_{t+\tau})\mathbf{1}_{\{t+\tau<\theta\}}\big] \nonumber \\
  &=\widehat{V}(t,x)+\E\left[\Lambda_{\tau\wedge(\theta-t)}\right]+\E\left[M_{\tau\wedge(\theta-t)}\right] \nonumber \\
  &\leq \widehat{V}(t,x)\label{eqn:verifbet}
\end{align}
where the last inequality follows from the optional sampling theorem upon noting that $\widehat{V}_x$ is bounded and hence $M$ is a true martingale.
Consequently, taking the supremum over all admissible stopping times in \eqref{eqn:verifbet} yields
  \[V(t,x)=\sup_{0\leq\tau\leq T-t}\E[X_{t+\tau}\mathbf{1}_{\{t+\tau<\theta\}}]\leq\widehat{V}(t,x).\]
To establish the reverse inequality note that, since $\widehat{V}(t+\tau^*,X_{t+\tau^*})=X_{t+\tau^*}$ and $\E[\Lambda_{\tau^*\wedge(\theta-t)}]=0$, both inequalities in \eqref{eqn:verifbet} become equalities for $\tau=\tau^*$.
Thus
  \[V(t,x)\geq\E[X_{t+\tau^*}\mathbf{1}_{\{t+\tau^*<\theta\}}]=\widehat{V}(t,x),\]
completing the proof.
\end{proof}

\end{document}